\documentclass[12pt,a4paper,reqno]{amsart}  
\usepackage{url}
\usepackage{fullpage}
\usepackage{enumerate} 
 
\usepackage[english]{babel} 

\newcommand{\eps}{\varepsilon}

\newcommand{\dx}{\mathrm{d}}
\newcommand{\si}{\operatorname{si}}

\newcommand{\R}{\mathbb{R}}
\newcommand{\Stilde}{\widetilde{S}}
\newcommand{\Ttilde}{\widetilde{T}}
\newcommand{\Rtilde}{\widetilde{R}}
\newcommand{\Jtilde}{\widetilde{J}}
\newcommand{\Etilde}{\widetilde{E}}

\newcommand{\Odip}[2]{\mathcal{O}_{#1}\!\left(#2\right)\mathchoice{\!}{}{}{}}
\newcommand{\Odig}[1]{\mathcal{O}\Bigl(#1\Bigr)\mathchoice{\!}{}{}{}}

\newcommand{\Odim}[1]{\mathcal{O}\bigl(#1\bigr)}
\newcommand{\Odi}[1]{\Odip{}{#1}}

\newcommand{\odip}[2]{{o}_{#1}\!\left(#2\right)\mathchoice{\!}{}{}{}}
\newcommand{\odi}[1]{\odip{}{#1}}

\allowdisplaybreaks

\renewcommand{\qedsymbol}{$\square$}
\newenvironment{Proof}[1][Proof]{\par\noindent\textbf{#1.}~}
{\hfill\qedsymbol\smallskip\par}

\newtheoremstyle{slantedtheorems}
{10pt}
{6pt}
{\slshape}
{}
{\bfseries}
{.}
{.5em}
{\thmname{#1} \thmnumber{#2}\thmnote{ (#3)}}

\theoremstyle{slantedtheorems}
\newtheorem{Theorem}{Theorem}
\newtheorem{Definition}{Definition}

\newtheorem{Lemma}{Lemma}

\newtheoremstyle{nonum}
{10pt}
{6pt}
{\slshape}
{}
{\bfseries}
{.}
{.5em}
{\thmname{#1}\thmnote{#3}}

\theoremstyle{nonum}
\newtheorem{nonumTheorem}{Theorem}

\begin{document} 

\title{Explicit relations between primes in short intervals
and exponential sums over primes}
\date{}
\author{Alessandro Languasco \lowercase{and} Alessandro Zaccagnini}
%
\begin{abstract}
Under the assumption of the Riemann Hypothesis (RH), 
we prove explicit quantitative relations between hypothetical error terms in the 
asymptotic formulae for truncated mean-square average of
exponential sums over primes and in the mean-square of primes 
in short intervals. We also remark that such relations are connected with a more
precise form of Montgomery's pair-correlation conjecture.
\end{abstract}
\maketitle

\section{Introduction}
In many circle method applications a key role is played by the asymptotic
behavior as $X\to\infty$ of the truncated mean square of the exponential 
sum over primes, i.e. by
\begin{equation*} 
R(X,\xi) =  \int_{-\xi}^\xi  \vert S(\alpha)-T(\alpha)\vert^2\, \dx \alpha,
\qquad 
\frac{1}{2X} \leq \xi\leq \frac{1}{2} ,
\end{equation*}
where  
$S(\alpha)= \sum_{n\leq X}\Lambda(n)e(n\alpha)$, 
$T(\alpha)=\sum_{n\leq X} e(n\alpha)$,
$e(x)=e^{2\pi i x}$ and $\Lambda(n)$ is the von Mangoldt function. 
In 2000  the first author and Perelli  \cite{LanguascoP2000a}
studied how to connect, under the assumption
of the Riemann Hypothesis (RH) and of Montgomery's
pair-correlation conjecture, the behaviour as $X\to\infty$
of $R(X,\xi)$  with the one of the mean-square of primes in short intervals, 
\textsl{i.e.}, with
\begin{equation*} 
J(X,h) = \int_1^X ( \psi(x+h) - \psi(x) - h )^2 \, \dx x,
\qquad
1\leq h \leq X,
\end{equation*}
where $\psi(x)=\sum_{n\leq x}\Lambda(n)$.
Recalling that Goldston and Montgomery \cite{GoldstonM1987}
proved that the asymptotic behavior of $J(X,h)$ as $X\to \infty$
is related with Montgomery's pair correlation function
\[
F(X,T)= 4 \sum_{0<\gamma,\gamma'\leq T} 
\frac{X^{i(\gamma-\gamma')}} {4+(\gamma-\gamma')^2},
\]
where   $\gamma,\gamma'$ run over the
imaginary part of the non-trivial zeros of the Riemann zeta function,
the following result was proved in \cite{LanguascoP2000a}.
\begin{nonumTheorem}
Assume RH. As $X\to\infty$, the following statements are equivalent:
\begin{enumerate}[(i)]
\item for every $\eps >0$, $R(X,\xi) \sim 2X\xi \log X\xi$
 uniformly for $X^{-1/2+\eps}\leq \xi \leq 1/2$;
\item for every $\eps >0$, $ J(X,h) \sim hX\log(X/h)$
 uniformly for $1\leq h \leq X^{1/2-\eps}$;
\item for every $\eps >0$ and $A\geq 1$, $ F(X,T) \sim (T/2\pi)\log
\min(X,T)$ uniformly for $X^{1/2+\eps} \leq T \leq X^A$.
\end{enumerate}
\end{nonumTheorem}

We remark that the uniformity ranges here are smaller than the
ones in \cite{GoldstonM1987} and that it 
is due to the presence of $E(X,h)$
(a term which naturally  comes from Gallagher's lemma), 
see \eqref{E(X,h)-def} and Lemma \ref{Lemma-Go}.

In 2003 Chan \cite{Chan2003} formulated a more precise pair-correlation
hypothesis and gave explicit results for the connections between the error 
terms in the asymptotic formulae for $F(X,T)$ and $J(X,h)$.
Such results were recently extended and improved by the authors of this paper 
in a joint work with Perelli \cite{LanguascoPZ2012a}: writing
\begin{align}
\label{F-asymp}
F(X,T)  &= \frac{T}{2\pi}\Bigl(\log\frac{T}{2\pi} - 1\Bigr) + R_F(X,T), \\
\label{J-asymp}
J(X,h) &=  h X\Bigl(\log \frac{X}{h} + c' \Bigr) + R_J(X,h)
\end{align}
and $c'=-\gamma-\log(2\pi)$ ($\gamma$ is Euler's constant), they gave explicit 
relations between \eqref{F-asymp} and \eqref{J-asymp} with error terms
essentially of type
\[
R_F(X,T) \ll \frac{T^{1-a}}{(\log T)^{b}}  \qquad \text{and} \qquad 
R_J(X,h) \ll \frac{h X}{(\log X)^{b}} \Bigl(\frac{h}{X}\Bigr)^{a} ,
\]
with $X$, $T$ and $h$ in suitable ranges and $a,b\geq 0$.

Our aim here is to prove explicit connections between the error terms
in the asymptotic formulae for $R(X,\xi)$ and $J(X,h)$ in the same
fashion of \cite{LanguascoPZ2012a}, but, recalling the 
previously cited theorem in \cite{LanguascoP2000a}, we have to restrict 
our attention to the range $1\leq h \leq X^{1/2-\eps}$ 
(or, equivalently, to $X^{-1/2+\eps}\leq \xi \leq 1/2$). 
In what follows the implicit constants may depend on $a,b$.
Our first result is 
\begin{Theorem}
\label{Th-anda} 
Assume RH and let $1 \leq h\leq X^{1/2 - \eps} $, 
$X^{-1/2 + \eps} \leq \xi\leq 1/2$.
Let further $0\leq a <1$, $b \geq 0$, $(a,b)\neq (0,0)$ be fixed.
If, for some constant $c\in \R$, we have
\begin{equation}
\label{Thanda-hypothesis}
R(X,\xi) = 2X\xi \log X\xi + c X\xi + 
\Odi{\frac{(X\xi)^{1-a}}{(\log X\xi)^{b}}},
\end{equation}
then
\begin{equation*} 
J(X,h) = hX \Bigl(\log\frac{X}{h} + c'\Bigr) 
+ \Odi{X+E(X,h) + R_{a,b}(X,h)},
\end{equation*}
provided that \eqref{Thanda-hypothesis} holds uniformly for 
\begin{equation}
\label{xi-h-cond}
\frac{1}{h}
\Bigl(\frac{h}{X}\Bigr)^{a}  (\log X)^{-b-4} 
\leq \xi \leq 
\frac{1}{h}
\Bigl(\frac{X}{h}\Bigr)^{a}  (\log X)^{b+4},
\end{equation} 
where 
\begin{equation}
\label{c'-def}
c' = c/2+2  - \gamma - \log(2\pi),
\end{equation} 
\begin{equation}
\label{R-a-b-def}
R_{a,b}(X,h)
=
\begin{cases}
h X \log\log X (\log X)^{-b}  & \text{if} \ a=0\\
h X  \bigl(h/X\bigr)^{a} (\log X)^{-b}  & \text{if} \ a>0,
\end{cases}
\end{equation}
and, for every fixed $\eps>0$, we define
\begin{equation}
\label{E(X,h)-def}
E(X,h) 
=
\begin{cases}
(h+1)^3 (\log X)^{2}  &\text{(uncond.) uniformly for}\ 0< h \leq X^{\eps} \\
h^3      &\text{(uncond.) uniformly for}\ X^{\eps}\leq h \leq X\\
(h+1)X (\log X)^{4}  &\text{(under RH) uniformly for}\ 0< h \leq X.
\end{cases}
\end{equation}
\end{Theorem}
We explicitly remark that, since $c'=-\gamma-\log(2\pi)$, by \eqref{c'-def} we
get $c=-4$ and that the conditions $\xi \leq 1/2$ and \eqref{xi-h-cond} imply
\[
h \gg X^{a/(a+1)}(\log X)^{(b+4)/(a+1)}
\]
which also leads to $R_{a,b}(X,h)\gg X$. It is also useful to remark that
the $E(X,h) \ll R_{a,b}(X,h)$ only for $h\ll X^{(1-a)/(2+a)} (\log
X)^{-b/(2+a)}$.

The technique used to prove Theorem \ref{Th-anda} is similar to the one in
Lemma~2 in \cite{LanguascoPZ2012a}; the main difference is in the 
presence of the terms $E(X,h)$ (which  comes from Lemma \ref{Lemma-Go})
and $\Odi{X}$ (which comes from the term $\Odi{1}$
in Lemma \ref{K-U-lemma}).

Concerning the opposite direction, we have
\begin{Theorem}
\label{Th-rianda} 
Assume RH and let $1 \leq h\leq X^{1/2 - \eps} $, $X^{-1/2 + \eps} \leq \xi\leq
1/2$.
Let further $0\leq a <1$, $b \geq 0$, $(a,b)\neq (0,0)$ be fixed.
If, for some constant $c'\in \R$, we have
\begin{equation}
\label{Thrianda-hypothesis}
J(X,h) = hX \Bigl(\log\frac{X}{h} + c'\Bigr) 
+ \Odig{\frac{hX}{(\log X)^{b}} \Bigl(\frac{h}{X}\Bigr)^{a}},
\end{equation}
then
\begin{equation}
\label{Thrianda-thesis}
R(X,\xi) = 2X\xi \log X\xi + c X\xi 
+ 
\Odi{\frac{(X \xi)^{3/(3+a)}}{(\log X)^{(b-a-2)/(3+a)}}
},
\end{equation}
provided that \eqref{Thrianda-hypothesis} holds uniformly for
\[
\frac{1}{\xi}
\frac{(X\xi)^{-a/(2a+6)}}{ (\log X)^{(a+b+4)/(2a+6)} }
\leq h \leq 
\frac{1}{\xi}
(X\xi)^{4a/(a+3)}  (\log X)^{(3a+4b+13)/(a+3)} ,
\] 
where $c = 2(c' - 2 + \gamma  + \log (2 \pi))$. 
\end{Theorem}

Note that for $a=0$ we have to take $b>2$ to get that the  error term in
\eqref{Thrianda-thesis} is $\odi{X\xi}$. 
The technique used to prove Theorem \ref{Th-rianda} is similar to the one in
Lemma 5 of \cite{LanguascoPZ2012a}; the main difference is in the 
use of Lemma \ref{Lem-RH-estim} which is needed to provide 
pair-correlation independent estimates of the involved quantities.

We remark that results similar to Theorems \ref{Th-anda}-\ref{Th-rianda}
can be proved using the weighted quantities
\[
\Stilde(\alpha)=\sum\limits_{n=1}^\infty\Lambda(n)e^{-n/X}e(n\alpha),
\quad
\Ttilde(\alpha) = \sum_{n=1}^\infty e^{-n/X}e(n\alpha), \quad
\]
\[
\Rtilde(X,\xi) = \int_{-\xi}^{\xi} \vert  
\Stilde(\alpha) - \Ttilde(\alpha) \vert ^2\, \dx \alpha ,
\quad
\Jtilde(X,h) = \int_0^\infty
( \psi(x+h) - \psi(x) - h )^2\,e^{-2x/X}\,\dx x
\]
in place of $S(\alpha)$, $T(\alpha)$,  $R(X,\xi)$ and  $J(X,h)$,
respectively.
The proofs are similar; in the analogue of Theorem \ref{Th-anda} 
the main difference is in using the second part of 
Lemma \ref{Lemma-Go} thus replacing  $E(X,h)$ with the sharper 
quantity $\Etilde(X,h)$ defined in \eqref{E-tilde-def}. 
Concerning the analogue of Theorem \ref{Th-rianda}, the key point is
in Eq.~\eqref{g-estim-PC}:  in this case we will be able to extend its range 
of validity to $\xi \leq x \leq \xi X^{1-\eps}$ and to get rid of the term 
$(x^3/\xi) (\log X)^{2}$.
These remarks lead to results which hold in wider ranges:
$1 \leq h\leq X^{1 - \eps}$ and $X^{-1 + \eps} \leq \xi\leq 1/2$.

The order of magnitude of $\Jtilde(X,h)$ can be directly
deduced from the one of $J(X,h)$ 
via  partial integration, see \emph{e.g.} eq.~\eqref{J-Jtilde}. 
Unfortunately, the vice-versa seems to be very hard to achieve;
this depends on the fact that we do not have sufficiently strong Tauberian
theorems to get rid of the exponential weight  in the 
definition of $\Jtilde(X,h)$.
Such a phenomenon is well known in the literature, see, \emph{e.g.},
Heath-Brown's remark on pages 385-386 of \cite{Heath-Brown1979d}.

\section{Some lemmas}

In the following we will need two weight functions.
\begin{Definition}
For  $h > 0$ we let
\begin{equation}
\label{K-U-def}
  K(\alpha, h)
  =
  \sum_{-h \le n \le h} (h - \vert n \vert) \, e(n \alpha)
  \qquad\text{and}\qquad
  U(\alpha, h)
  =
  \Bigl( \frac{\sin(\pi h \alpha)}{\pi \alpha} \Bigr)^2.
\end{equation}
\end{Definition}

We will need some information about the total mass of such weights.
\begin{Lemma}
\label{K-U-lemma}
For  $h > 0$, we have
\[
\int_{0}^{1/2} K(\alpha,h) \, \dx \alpha = \frac{h}{2}
 \qquad\text{and}\qquad
\int_{0}^{+\infty}  U ( \alpha, h)\, \dx \alpha = \frac{h}{2}.
\]
Moreover we also have
\begin{align*}
  \int_0^{1 / 2} \log(h \alpha) \, K(\alpha, h) \, \dx \alpha
  &=
  - \frac{h}{2} (\log(2 \pi) + \gamma - 1)
  +
  \Odi{1},
  \\
\int_{0}^{+\infty} \log(h \alpha) \, U(\alpha, h)\,  \dx \alpha 
 &= 
- \frac{h}{2}  (\log (2\pi)+ \gamma -1).
\end{align*}
\end{Lemma}

Before the proof, we remark that this lemma is consistent with the constant
in Lemma~2 of Languasco, Perelli and Zaccagnini
\cite{LanguascoPZ2012a}, taking into account the fact that our
variable $h$ here corresponds to $\pi \kappa$ there.

\begin{proof}
The results on $U(\alpha, h)$ can be immediately obtained
by integrals n.3.821.9 and  n.4.423.3, respectively  on pages 460
and 594 of Gradshteyn and Ryzhik \cite{GradshteynR2007}.

Now we prove the part concerning $K(\alpha, h)$.
The first identity immediately follows by isolating the contribution
of $n=0$ in the definition of $K(\alpha,h)$ and making a trivial computation.
To prove the second identity,
separating again the contribution of the term $n = 0$ and using
standard properties of the complex exponential functions, we have
\begin{align*}
  I(h)  
  &:=
  2 \int_0^{1 / 2} \log(h \alpha) \, K(\alpha, h) \, \dx \alpha \\
 & =
  2 h \int_0^{1 / 2} \log(h \alpha) \, \dx \alpha
  +
  4 \sum_{1 \le n \le h} (h - n)
    \int_0^{1 / 2} \log(h \alpha) \, \cos(2 \pi n \alpha) \, \dx \alpha \\
  &=
  h \log h - h (\log 2 + 1)
  +
  2
  \sum_{1 \le n \le h} (h - n)
    \int_0^1
      \log \Bigl(\frac{h \beta}{2} \Bigr) \, \cos(\pi n \beta) \, \dx \beta.
\end{align*}
We remark that
\[
  \int_0^1
    \log \Bigl(\frac{h}{2}  \Bigr) \, \cos(\pi n \beta) \, \dx \beta
  =
  0
\]
whenever $n$ is a positive integer, and hence we can write
\begin{align*}
  I(h)
  &=
  h \log h - h (\log 2 + 1)
  +
  2
  \sum_{1 \le n \le h} (h - n)
    \int_0^1
      \log\beta \, \cos(\pi n \beta) \, \dx \beta \\ 
&=
  h \log h - h (\log 2 + 1)
  -
  \sum_{1 \le n \le h} \frac{h - n}{n}
  -
  2
  \sum_{1 \le n \le h} (h - n) \frac{\si(\pi n)}{\pi n},
\end{align*}
by Formula 4.381.2 on page 581 of \cite{GradshteynR2007},
where the sine integral function is defined by
\begin{equation}
\label{si-def}
  \si(x)
  =
  - \int_x^{+ \infty} \frac{\sin t }t \, \dx t
\end{equation}
for $x > 0$. 
The elementary relation
\(
  \sum_{1 \le n \le h} 1/n  =
  \log h + \gamma + \Odi{h^{-1}}
\)
shows that
\begin{align*}
  I(h) 
  &=
  -
  h (\log 2 + \gamma) + \Odi{1}
  -
  \frac{2 h}\pi
  \sum_{1 \le n \le h} \frac{\si(\pi n)}n
  +
  \frac2\pi
  \sum_{1 \le n \le h} \si(\pi n).
\end{align*}
Finally we remark that Eq.~\eqref{si-def}  implies, by means of a
simple integration by parts, that $\si(x) \ll x^{-1}$ as
$x \to + \infty$.
Hence
\[
  \sum_{1 \le n \le h} \frac{\si(\pi n)}n
  =
  \sum_{n \ge 1} \frac{\si(\pi n)}n
  +
  \Odi{h^{-1}}
  =
  \frac\pi2 (\log \pi  - 1)
  +
  \Odi{h^{-1}},
\]
by Formula 6.15.2 on page 154 of \cite{OlverLBC2010}.
Moreover, by a double partial integration, we get
 \[
\si(x) = - \frac{\cos x}{x} - \frac{\sin x }{x^{2}}+ 2 \int_{x}^{+\infty}
\frac{\sin t }{t^{3}}\, \dx t
\]
and hence 
 \[
 \sum_{1 \le n \le h} \si(\pi n) =
  \sum_{1 \le n \le h}  
   \frac{(-1)^{n+1}}{\pi n}  +  
     \Odig{\sum_{1 \le n \le h}  \frac{1}{n^2} }
     \ll 1.
\]
In conclusion
\[
  I(h)
  =
    -
  h (\log 2 + \gamma)
  -
  \frac{2 h}\pi
  \Bigl(
    \frac\pi2 (\log\pi - 1)
    +
    \Odi{h^{-1}}
  \Bigr)
  +
  \Odi{1},
\]
and  Lemma \ref{K-U-lemma} is proved.
\end{proof}

Now we see some information about the order of magnitude
of $K(\alpha,h)$ and its first derivative.

\begin{Lemma}
\label{K'-lemma}
For $h \ge 1$ we have
\[
  K(\alpha,h)\ll \min \Bigl(h^2,\Vert \alpha \Vert^{-2} \Bigr),
\]
and
\[
  \frac{\dx}{\dx \alpha} K(\alpha, h)
  \ll
  h \Vert \alpha \Vert \min\Bigl( h^3 , \Vert \alpha \Vert^{-3} \Bigr).
\]
\end{Lemma}

\begin{proof}
We assume that $\alpha \in (0, 1 / 2)$ as we may.
We let $h_0 = [h]$. We first remark that
\begin{align*}
K(\alpha,h) 
&=
\sum_{-h_0 \le n \le h_0} (h_0 - \vert n \vert) \, e(n \alpha)
+
\{h\} \sum_{-h_0 \le n \le h_0} \, e(n \alpha)
\\
&=
\Bigl( \frac{\sin(\pi h_0 \alpha)}{\sin(\pi \alpha)} \Bigr)^2
+
\{h\} \sum_{-h_0 \le n \le h_0} \, e(n \alpha).
\end{align*}
Recalling the identity
\begin{equation}
\label{identity} 
  \sum_{-h_0 \le n \le h_0} e(n \alpha)
  =
  1
  +
  2
  \frac{\sin(\pi h_0 \alpha)}{\sin(\pi \alpha)}
  \cos\bigl( \pi (h_0 + 1) \alpha \bigr)
\end{equation}
and the estimate
\begin{equation}
\label{est-ratio}
  \frac{\sin(\pi h_0 \alpha)}{\sin(\pi \alpha)}
  \ll
  \min\bigl( h_0, \alpha^{-1} \bigr),
\end{equation}
the first part of the lemma immediately follows.

For the second inequality we first remark that
\begin{align*}
  \frac{\dx}{\dx \alpha} K(\alpha, h)
  &=
  2 \pi i
  \sum_{-h \le n \le h} (h - \vert n \vert) \, n \, e(n \alpha) \\
  &=
  2 \pi i
  \sum_{-h_0 \le n \le h_0} (h_0 - \vert n \vert) \, n \, e(n \alpha)
  +
  2 \pi i \{ h \}
  \sum_{-h_0 \le n \le h_0} n \, e(n \alpha) \\
  &=
  A(\alpha, h)
  +
  B(\alpha, h),
\end{align*}
say.
By \eqref{identity} we get that
\[
  A(\alpha, h)
  =
  \frac{\dx}{\dx \alpha}
    \Bigl( \frac{\sin(\pi h_0 \alpha)}{\sin(\pi \alpha)} \Bigr)^2
  \qquad\text{and}\qquad
  B(\alpha, h)
  =
  2 \{ h \}
  \frac{\dx}{\dx \alpha}
    \Bigl(
      \frac{\sin(\pi h_0 \alpha)}{\sin(\pi \alpha)}
      \cos\bigl( \pi (h_0 + 1) \alpha \bigr)
    \Bigr),
\]
respectively.
We remark that
\begin{equation}
\label{deriv}
  \frac{\dx}{\dx \alpha}
    \frac{\sin(\pi h_0 \alpha)}{\sin(\pi \alpha)}
  =
    \frac{\pi h_0 \cos(\pi h_0 \alpha) \sin(\pi \alpha)
          -
          \pi \sin(\pi h_0 \alpha)\cos(\pi \alpha)}
         {(\sin(\pi \alpha))^2}.
\end{equation}
For $\alpha \le h_0^{-1}$ a standard development shows that the
numerator in \eqref{deriv} is $\ll \alpha^3 h_0^3$, while the
denominator is $\gg \alpha^2$.
For $\alpha \in [h_0^{-1}, 1 / 2]$ it is easy to see that the
right-hand side of \eqref{deriv} is $\ll h_0 \alpha^{-1}$.
Summing up, we have
\begin{equation}
\label{est-deriv}
  \frac{\dx}{\dx \alpha}
    \frac{\sin(\pi h_0 \alpha)}{\sin(\pi \alpha)}
  \ll
  \min\bigl( \alpha h_0^3, \alpha^{-1} h_0 \bigr).
\end{equation}
A straightforward computation reveals that
\[
  \frac{\dx}{\dx \alpha}
    \Bigl( \frac{\sin(\pi h_0 \alpha)}{\sin(\pi \alpha)} \Bigr)^2
  =
  2
  \frac{\sin(\pi h_0 \alpha)}{\sin(\pi \alpha)}
  \frac{\dx}{\dx \alpha}
    \frac{\sin(\pi h_0 \alpha)}{\sin(\pi \alpha)}
  \ll
  \min\bigl( \alpha h_0^4, \alpha^{-2} h_0 \bigr),
\]
by \eqref{est-deriv} and \eqref{est-ratio}.
Furthermore
\begin{align*}
  \frac{\dx}{\dx \alpha}
    \Bigl(
      \frac{\sin(\pi h_0 \alpha)}{\sin(\pi \alpha)}
      \cos\bigl( \pi (h_0 + 1) \alpha \bigr)
    \Bigr)
  &=
  \frac{\dx}{\dx \alpha}
  \Bigl(
    \frac{\sin(\pi h_0 \alpha)}{\sin(\pi \alpha)}
  \Bigr)
  \cos\bigl( \pi (h_0 + 1) \alpha \bigr) \\
  &\qquad-
  \pi (h_0 + 1)
  \frac{\sin(\pi h_0 \alpha)}{\sin(\pi \alpha)}
  \sin\bigl( \pi (h_0 + 1) \alpha \bigr),
\end{align*}
and a similar computation yields
\[
   B(\alpha, h)
  \ll
  \min \bigl( \alpha h_0^3, \alpha^{-1} h_0 \bigr),
\]
which is of lower order of magnitude.
Hence the second part of Lemma \ref{K'-lemma} is proved.
\end{proof}
We also remark that estimates similar to the ones in 
Lemma \ref{K'-lemma} hold for $U(\alpha, h)$ too; since
they immediately follow from the definition we
omit their proofs.

Let now
\[
\widehat{f}(t)=\int_{-\infty}^{+\infty}f(x) e(-tx)\, \dx x
\]
be the Fourier transform of $f(x)$. 
We need the following auxiliary result which is based on
Gallagher's lemma.
\begin{Lemma}
\label{Lemma-Go}
Let $0<h\leq X$,
\begin{equation}
\label{R-tildeR-def}
R(\alpha) = S(\alpha) - T(\alpha) 
\quad \text{and} \quad
\Rtilde(\alpha) = \Stilde(\alpha) - \Ttilde(\alpha).
\end{equation} 
Then
\[
\int_{-1/2}^{1/2} 
  \vert R(\alpha) \vert^2\,   K(\alpha,h) \, \dx \alpha
=
\int_{-\infty}^{+\infty} \vert R(\alpha) \vert^2 \, U(\alpha,h)\, \dx \alpha
= J(X,h) + \Odi{E(X,h)},
\]
where $E(X,h)$ is defined in \eqref{E(X,h)-def}. Moreover we have,
\[
\int_{-1/2}^{1/2} 
  \vert \Rtilde(\alpha) \vert^2\,   K(\alpha,h) \, \dx \alpha
=
\int_{-\infty}^{+\infty} \vert \Rtilde(\alpha) \vert^2\, U(\alpha,h)\, \dx
\alpha
= \Jtilde(X,h) + \Odim{\Etilde(X,h)},
\]
where, for every fixed $\eps>0$, we define
\begin{equation}
\label{E-tilde-def}
\Etilde(X,h) 
=
\begin{cases}
(h+1)^3  (\log X)^{2}  &\text{(uncond.) uniformly for}\ 0<  h \leq  X^{\eps}\\
h^3   &\text{(uncond.) uniformly for}\   X^{\eps}< h \leq X\\
(h+1)^{2}  (\log X)^{4}  &\text{(under RH) uniformly for}\ 0< h \leq X .
\end{cases}
\end{equation}
\end{Lemma}

\begin{Proof}
The first part is  Lemma 1 of \cite{LanguascoP2000a}, so we skip the proof.
For the second part, we start  remarking that 
Lemma 1.9 of Montgomery \cite{Montgomery1971} gives
\begin{equation}
\label{Gall-0}
\int_{-\infty}^{+\infty} 
\vert \Rtilde(\alpha) \vert^2 \, U(\alpha,h) \, \dx \alpha 
=
\int_{-\infty}^{+\infty}
\vert \sum_{\substack{\vert n-x\vert < h/2 \\ n\geq 1}} (\Lambda(n)-1)
e^{-n/X}\vert ^2 \ \dx x.
\end{equation}
By periodicity  we have
\[
\int_{-\infty}^{+\infty}  
\vert \Rtilde(\alpha) \vert^2 \, U(\alpha,h)\, \dx \alpha 
=
\int_{-1/2}^{1/2} 
\vert \Rtilde(\alpha) \vert^2   
\Bigl(\sum_{n=-\infty}^{+\infty} U(n+\alpha,h)\Bigr)\, \dx \alpha.
\]
Since $\widehat{U}(\alpha,h) = \max(h-\vert \alpha \vert ; 0)$,
by Poisson's summation formula  and \eqref{K-U-def} we get
\[ 
\sum_{n=-\infty}^{+\infty}
U(n+\alpha,h)
=
\sum_{n=-\infty}^{+\infty}
\widehat{U}(\alpha,h) e(n\alpha) 
=  
K(\alpha,h),
\]
and hence, using \eqref{Gall-0}, we obtain
\begin{align}
\notag
\int_{-1/2}^{1/2} 
\vert \Rtilde(\alpha) \vert^2   \
&K(\alpha,h)\,  \dx \alpha 
=
\int_{-\infty}^{+\infty}
\vert \sum_{\substack{x <n \leq x + h \\ n\geq 1}}
(\Lambda(n)-1) e^{-n/X}\vert ^2 \ \dx x
\\
\label{gall-01}
&=
\int_{0}^{+\infty}
\vert \sum_{x <n \leq x + h} 
(\Lambda(n)-1) e^{-n/X}\vert ^2 \ \dx x
+\Odi{(h+1)^{2}  (\log (h+1))^{4} },
\end{align}
where in the last estimate we assumed RH and we used
\begin{equation}
\label{psi-RH-estim}
\psi(y)=y+\Odi{y^{1/2} (\log y)^{2} }
\end{equation} 
on a interval of length $\leq h$.
Noting that
\begin{align*}
  \sum_{x < n \le x + h} (\Lambda(n) - 1) e^{-n / X}
  &=
  e^{-x / X} (\psi(x + h) - \psi(x) - h)
  \Bigl(1 + \Odig{\frac{h+1}{X}} \Bigr)
\end{align*}
and recalling that $h\leq X $, from \eqref{gall-01} we have
\begin{equation*} 
\int_{-1/2}^{1/2} 
\vert \Rtilde(\alpha) \vert^2  \,
K(\alpha,h)\, \dx \alpha
=
\Jtilde(X,h) 
\Bigl(
1 + \Odig{\frac{h+1}{X}}
\Bigr)
+\Odi{(h+1)^{2}  (\log X)^{4} }.
\end{equation*}
%

To estimate the last error term we  connect $\Jtilde(X,h)$
to $J(X,h)$. A partial integration immediately gives
\begin{equation}
\label{J-Jtilde}
\Jtilde(X,h)
  =
  \frac2X
  \int_0^{\infty} J(t, h) e^{-2t/X} \, \dx t.
\end{equation}
To estimate the right-hand side of \eqref{J-Jtilde},
we split the range of integration into
$[0, h] \cup [h, +\infty)$.
A direct computation using  \eqref{psi-RH-estim} shows that
\[
  \int_0^h J(t, h) e^{- 2 t / X} \, \dx t
  \ll
  h  (\log h)^{4} 
  \int_0^h t \ e^{- 2 t / X} \, \dx t
  \ll
  h^3  (\log h)^{4} .
\]
Still assuming RH, the Selberg \cite{Selberg1943} estimate  
gives,  for $1\leq h \leq t$, that
\begin{equation}
\label{Selberg-J-estim}
J(t,  h) \ll ht (\log t)^{2} 
\end{equation}
and so we get
\[
  \int_h^{+\infty} J(t, h) e^{- 2 t / X} \, \dx t
  \ll
  h \int_h^{+\infty} t  (\log t)^{2} \ e^{- 2 t / X}  \, \dx t
    \ll
  h X^2  (\log X)^{2} .
\]
%
Summing up, under RH we have
\begin{equation*}
\Jtilde(X,  h) \ll (h+1)X (\log X)^{4} 
\end{equation*}
%
we can finally write 
\begin{align*}
\int_{-1/2}^{1/2} 
\vert \Rtilde(\alpha) \vert^2   \,
K(\alpha,h)\, \dx \alpha
&=
\Jtilde(X,h)
+ 
\Odi{(h+1)^{2} (\log X)^{4} }.
\end{align*}
%
%
The unconditional cases follow by replacing \eqref{psi-RH-estim} with
the Brun-Titchmarsh inequality and \eqref{Selberg-J-estim}
with the Lemma in \cite{Languasco1998}.  
\end{Proof}

In the next sections we will also need the following remark.
Let $\xi>0$ and $\delta \xi =1/2$.
In this case $U(\alpha, \delta) \gg \delta^{2}$ for 
$\vert \alpha\vert \leq \xi$; hence
by the first equation in Lemma \ref{Lemma-Go} we obtain
\begin{equation}
\label{RH-R-estim-a}
\int_{-\xi}^{\xi} \vert R(\alpha) \vert^2 \, \dx \alpha
\ll 
\xi^{2}  
\Bigl( J\Bigl(X, \frac{1}{2\xi}\Bigr) +  E\Bigl(X,\frac{1}{2\xi}\Bigr)\Bigr).
\end{equation}
By \eqref{Selberg-J-estim} and \eqref{E(X,h)-def}, under RH we immediately
obtain, for every $1/(2X) \leq \xi \leq 1/2$, that
\begin{equation}
\label{RH-R-estim-b}
\int_{-\xi}^{\xi} \vert R(\alpha) \vert ^{2} \ \dx \alpha 
\ll X \xi (\log X)^{4}.
\end{equation}

\section{Proof of Theorem \ref{Th-anda}}

We use Lemma \ref{Lemma-Go} in the form
\begin{equation}
\label{startpoint}
J(X,h) = 
\int_{-1/2}^{1/2}  \vert R(\alpha) \vert^2\,  K(\alpha,h) \, \dx \alpha 
+ 
\Odi{E(X,h)},
\end{equation}
where $R(\alpha)$ is defined in \eqref{R-tildeR-def}.
Observe that both $ \vert R(\alpha) \vert^2$ and $K(\alpha,h)$ are
even functions of $\alpha$, and hence we may restrict our attention to
$\alpha\in [0,1/2]$.  
Writing
\begin{equation}
\label{f-def}
f(X,\alpha) = X\log (X\alpha) + \bigl(\frac{c}{2}+1\bigr) X
= X\log \frac{X}{h} + X\log (h\alpha) + \bigl(\frac{c}{2}+1\bigr) X,
\end{equation}
we can approximate $\vert R(\alpha) \vert^2$ as  
$ \vert R(\alpha) \vert^2 =  f(X,\alpha) + 
\bigl( \vert R(\alpha) \vert^2 - f(X,\alpha) \bigr)$.
Using Lemma \ref{K-U-lemma} and \eqref{f-def},
we obtain
\begin{equation}
\label{main-term}
\int_{0}^{1/2}  f(X,\alpha) K(\alpha,h) \, \dx \alpha 
=
\frac{h}{2} X\log \frac{X}{h}+ c' \frac{h}{2} X + \Odi{X},
\end{equation}
where $c'$ is defined in \eqref{c'-def}.

Let now $U_{1}<1/h <U_{2} \leq 1$ be two parameters to be chosen later. 
Hence by Lemma \ref{K'-lemma} and \eqref{RH-R-estim-b} we immediately obtain
\begin{align}
\notag
 \int_{0}^{U_{1}} 
\bigl( \vert R(\alpha) \vert^2 -  f(X,\alpha) \bigr) K(\alpha,h) \, \dx \alpha 
&\ll
h^{2} \int_{0}^{U_{1}}  \vert R(\alpha) \vert^2 \, \dx \alpha 
+
 h^{2} \int_{0}^{U_{1}} f(X,\alpha)\, \dx \alpha 
\\
\label{coda-bassa-estim}
&
\ll
h^{2} U_{1} X  (\log X)^{4} .
\end{align}
Again by Lemma \ref{K'-lemma} and \eqref{RH-R-estim-b},
by partial integration we have
\begin{align}
\notag
 \int_{U_{2}}^{1/2} 
\bigl( \vert R(\alpha) \vert^2 - f(X,\alpha) \bigr) K(\alpha,h) \, \dx \alpha 
&
\ll
\int_{U_{2}}^{1/2} \frac{\vert R(\alpha) \vert^2}{\alpha^{2}} \, \dx \alpha 
 +  
\int_{U_{2}}^{1/2} \frac{f(X,\alpha)}{\alpha^{2}} \, \dx \alpha 
\\
\label{coda-alta-estim}
&
\ll
\frac{X  (\log X)^{4} }{U_{2}}.
\end{align}
From \eqref{coda-bassa-estim}-\eqref{coda-alta-estim}
it is clear that the optimal choice is 
$h^{2} U_{1} = 1/U_{2}$.
We now evaluate
\[
\int_{U_{1}}^{U_{2}} 
\bigl( \vert R(\alpha)\vert^2- f(X,\alpha) \bigr) K(\alpha,h) \, \dx \alpha .
\]
A direct computation and the hypothesis show that
\[
\int_{0}^{\xi} 
\bigl( \vert R(\alpha)\vert^2- f(X,\alpha) \bigr)  \, \dx \alpha 
\ll 
\frac{(X \xi)^{1-a}} { (\log X\xi)^{b}} ,
\]
and hence, by partial integration and Lemma \ref{K'-lemma}, we obtain
\begin{align*}
\int_{U_{1}}^{U_{2}} 
\bigl( \vert R(\alpha)\vert^2- f(X,\alpha) \bigr) K(\alpha,h) \, \dx \alpha 
&\ll 
h^2 \frac{(X U_1)^{1-a}}{(\log X)^{b}} 
+
\frac{X^{1-a} U_2^{-1-a}}{(\log X)^{b}} 
\\
&+
\frac{h X^{1-a}}{(\log X)^{b}} 
 \int_{U_{1}}^{U_{2}} 
 \xi^{2-a}  \min\bigl( h^3, \xi^{-3} \bigr) \, \dx \xi .
\end{align*}
Using the constraints $h^{2} U_{1} = 1/U_{2}$ and $U_1<1/h$,
the right-hand side is 
\begin{align*}
\ll
\frac{h^{1+a} X^{1-a}}{(\log X)^{b}}  
 +
\frac{h X^{1-a}}{(\log X)^{b}}  
 \int_{1/h}^{U_{2}} 
 \xi^{-1-a}   \, \dx \xi 
 \ll
\frac{h^{1+a} X^{1-a}}{(\log X)^{b}} 
+
R_{a,b}(X,h,U_2) ,
\end{align*}
where 
\[
R_{a,b}(X,h,U_2)
=
\begin{cases}
h X \log(hU_2) (\log X)^{-b}  & \text{if} \ a=0\\
h^{1+a}X^{1-a}$ $ (\log X)^{-b}  & \text{if} \ a>0.
\end{cases}
\]
Combining such results we get
\begin{align}
\label{center-estimate}
\int_{U_{1}}^{U_{2}} 
\bigl( \vert R(\alpha)\vert^2- f(X,\alpha) \bigr) K(\alpha,h) \, \dx \alpha 
\ll
R_{a,b}(X,h,U_2).
\end{align} 
Hence, by \eqref{coda-bassa-estim}-\eqref{center-estimate} and $h^{2} U_{1} =
1/U_{2}$ we get
\begin{equation}
\label{final-error}
\int_{0}^{1/2} 
\bigl( \vert R(\alpha) \vert^2 - f(X,\alpha) \bigr) K(\alpha,h) \, \dx \alpha 
\ll 
\frac{X  (\log X)^{4} }{U_{2}}
+
R_{a,b}(X,h,U_2).
\end{equation}
Choosing 
\[
U_2= \frac{X^a (\log X)^{b+4} }{h^{1+a}}   
\quad \text{and} \quad
U_1= \frac{h^{a-1}}{X^a (\log X)^{b+4}} ,
\]
by \eqref{main-term} and \eqref{final-error} we finally get
\[
\int_{0}^{1/2} 
\vert R(\alpha) \vert^2\,  K(\alpha,h) \, \dx \alpha 
=
\frac{h}{2} X\log \frac{X}{h}+ c' \frac{h}{2} X + 
\Odi{X + R_{a,b}(X,h)}
\]
where $c'$ and $R_{a,b}(X,h)$ are defined in 
\eqref{c'-def} and \eqref{R-a-b-def}.
Theorem \ref{Th-anda} follows from \eqref{startpoint}.

\section{Proof of Theorem \ref{Th-rianda}}

We adapt the proof of Lemma 5 of \cite{LanguascoPZ2012a} (which is an explicit
form of Lemma 4 of \cite{GoldstonM1987}). We recall that
$0<\eta<1/4$ is a parameter to be chosen later and
\[
K_\eta(x) 
= \frac{\sin (2\pi x) + \sin (2\pi (1+\eta)x)} {2\pi x(1-4\eta^2 x^2)},
\]
so that
\[
\widehat{K}_\eta(t) = 
\begin{cases}
1 & \text{if}\  \vert t \vert  \leq 1\\
\cos^2\Bigl(\dfrac{\pi( \vert t \vert -1)}{2\eta}\Bigr)& \text{if}\ 1\leq \vert
t \vert \leq 1+\eta\\
0 & \text{if}\  \vert t \vert \geq 1+\eta
\end{cases}
\]
and
\begin{equation}
\label{K''-eta-estim}
K''_\eta(x) \ll \min(1;(\eta x)^{-3}),
\end{equation}
see Eqs.~(3.14)-(3.15) and Lemma 4 of  \cite{LanguascoPZ2012a}.
Moreover, by Lemma 3 of \cite{LanguascoPZ2012a},  we also have
\begin{equation}
\label{K-transf}
\widehat{K}_\eta(t) 
=
\int_{0}^{\infty} K''_\eta(x)  U(t,x) \, \dx x.
\end{equation}

Hence, again considering only positive values of $\alpha$, we have
\begin{equation}
\label{bounds-1} 
\int_{0}^{\infty}  \vert R(\alpha) \vert^2\, 
\widehat{K}_\eta\Bigl( \frac{ \alpha}{\xi} (1+\eta)\Bigr)  \, \dx \alpha 
\leq 
\frac{R(X,\xi)}{2}  
\leq
\int_{0}^{\infty}  \vert R(\alpha) \vert^2\, 
\widehat{K}_\eta\Bigl( \frac{ \alpha}{\xi}\Bigr) \, \dx \alpha
\end{equation}
where $R(\alpha)$ is defined in \eqref{R-tildeR-def}.
Writing $f(X,\alpha)$ as in \eqref{f-def},  
we approximate $\vert R(\alpha) \vert^2$ as
$
\vert R(\alpha) \vert^2= f(X,\alpha) + ( \vert R(\alpha) \vert^2-f(X,\alpha) ).
$
Observing that $U(\alpha/\xi,x) = \xi^2 U(\alpha, x /\xi)$,
letting
\begin{equation*} 
g(x,\xi)
=
\xi^2
\int_0^\infty 
( \vert R(\alpha) \vert^2 - f(X,\alpha))
U\Bigl(\alpha, \frac{x}{\xi} \Bigr) \ \dx \alpha 
\end{equation*}
and using \eqref{K-transf}, we get
\begin{equation}
\label{estim-1} 
\int_{0}^{\infty}  \vert R(\alpha) \vert^2\, 
\widehat{K}_\eta\Bigl( \frac{ \alpha}{\xi}\Bigr) \, \dx \alpha
= 
\int_{0}^{\infty} f(X,\alpha) 
\widehat{K}_\eta\Bigl( \frac{ \alpha}{\xi}\Bigr) \, \dx \alpha 
+ 
\int_{0}^{\infty} K''_\eta(x) g(x,\xi)
 \, \dx x 
= J_1 +J_2,
\end{equation}
say. 
A direct computation shows that
\begin{equation}
\label{J1-behav}
J_1 =  X\xi \log X\xi + \frac{c}{2} X\xi  + \Odi{\eta X\xi \log X\xi} .
\end{equation}
In order to estimate $J_2$ we first remark that by Lemma \ref{K-U-lemma}, 
\eqref{f-def} and \eqref{c'-def}, we have
\begin{equation}
\label{internal}
\xi^2
\int_{0}^{\infty} f(X,\alpha)
U \Bigl( \alpha,\frac{x}{\xi}\Bigr)\, \dx \alpha 
= 
\frac{x X \xi}{2} \log \frac{X\xi}{x} +   \frac{c'}{2} x X \xi.
\end{equation}

Now we need the following
\begin{Lemma}
\label{Lem-RH-estim} 
Assume RH and let $\eps>0$. 
We have
\begin{equation}
\label{g-estim-RH}
g(x,\xi)
\ll 
\begin{cases}
X\xi^2\log X & \text{if}\ 0 < x \leq \xi  
\\
x X\xi (\log X)^{2}  &\text{if} \ \xi \leq x \leq \xi X^{1/2-\eps}
\\ 
x X\xi  (\log X)^{4}  & \text{if} \ x \geq  \xi X^{1/2-\eps}.
\end{cases}
\end{equation} 
Assume further the hypothesis of Theorem \ref{Th-rianda}. 
We have 
\begin{equation}
\label{g-estim-PC}
g(x,\xi)
\ll 
x^{1+a}\frac{(X\xi)^{1-a}}{ (\log X)^{b} } + \frac{x^3}{\xi} (\log X)^{2}
\quad \text{if}\quad
\xi \leq x \leq \xi X^{1/2-\eps}.
\end{equation}
\end{Lemma}
\begin{Proof} 
Let $h= x/\xi$.
We first prove \eqref{g-estim-RH}.
For $0<h\leq 1$ we have 
$U(\alpha,h) \ll \min (1;\alpha^{-2})$ 
and hence by periodicity
\[
\int_0^\infty  \vert R(\alpha) \vert^2\, U(\alpha,h) \, \dx \alpha 
\ll 
\sum_{n=1}^\infty \frac{1}{n^2}\int_{n-1}^n \vert R(\alpha) \vert^2\, \dx \alpha\ll 
T\Bigl(X,\frac{1}{2}\Bigr)+S\Bigl(X,\frac{1}{2}\Bigr)
\ll X\log X,
\]
by the Prime Number Theorem (this case is, in fact, unconditional).
For $1\leq h \leq X^{1/2-\eps}$ the assertion follows immediately from
\eqref{RH-R-estim-a}, \eqref{Selberg-J-estim}, \eqref{E(X,h)-def} and
\eqref{internal}.
Finally, for $h \geq   X^{1/2-\eps}$ we use \eqref{psi-RH-estim}
after having applied
Gallagher's lemma, see Lemma 1.9 of
Montgomery \cite{Montgomery1971}, which gives
\begin{equation*} 
\int_{-\infty}^{+\infty} \vert R(\alpha)\vert ^2\, U(\alpha,h)\ \dx \alpha =
\int_{-\infty}^{+\infty}
\vert \sum_{\substack{\vert n-x\vert <h/2 \\ 1\leq n\leq X}}
(\Lambda(n)-1)\vert^2\ \dx x.
\end{equation*}
Using  \eqref{internal}, this case holds true.
We now prove \eqref{g-estim-PC}.
For $1\leq h \leq X^{1/2-\eps}$ the assertion follows immediately
from \eqref{internal}, Lemma \ref{Lemma-Go} and the hypothesis of Theorem
\ref{Th-rianda}.
\end{Proof}

Choosing now $V_1, V_2$ such that $\xi< V_1 < 1/\eta < V_2 < \xi X^{1/2-\eps}$,
we split $J_{2}$'s integration range into six subintervals. 
We obtain 
\begin{align}
\notag
J_{2}
&=
\Bigl(
\int_{0}^{\xi} 
+
\int_{\xi}^{V_1}
+
\int_{V_1}^{1/\eta}
+
\int_{1/\eta}^{V_2}
+
\int_{V_2}^{\xi X^{1/2-\eps}}
+
\int_{\xi X^{1/2-\eps}}^{+\infty}
\Bigr)
K''_\eta(x) g(x,\xi)  \, \dx x 
\\
\label{J2-split}
&=
M_{1}+M_{2}+M_{3}+M_{4}+M_{5}+M_{6},
\end{align}
say.
By Lemma \ref{Lem-RH-estim} and \eqref{K''-eta-estim}, we obtain
\[
M_{1} \ll 
X \xi^2 \log X
\int_{0}^{\xi} \dx x 
\ll
 X \xi^{3} \log X,
\]
\[
M_{2} \ll 
X\xi (\log X)^{2} 
\int_{\xi}^{V_1} x\, \dx x 
\ll
X\xi V_1^2 (\log X)^{2} ,
\]
\[
M_{3} \ll 
\int_{V_1}^{1/\eta}
\Bigl(
x^{1+a}\frac{(X\xi)^{1-a}}{ (\log X)^{b} }
+
\frac{x^3}{\xi} (\log X)^{2} 
\Bigr)
\, \dx x 
\ll
\frac{(X \xi)^{1-a}}{\eta^{2+a} (\log X)^{b} }
+
\frac{ (\log X)^{2} }{\xi \eta^{4}},
\]
\[
M_{4} \ll 
\frac{1}{\eta^{3}}
\int_{1/\eta}^{V_2}
\Bigl(
x^{a-2}\frac{(X\xi)^{1-a}}{ (\log X)^{b} }
+
\frac{ (\log X)^{2} }{\xi}
\Bigr)\, \dx x
\ll
\frac{(X \xi)^{1-a}}{\eta^{2+a} (\log X)^{b} }
+
\frac{V_2  (\log X)^{2} }{\xi \eta^{3}},
\]
\[
M_{5} \ll 
\frac{X\xi (\log X)^{2} }{\eta^{3}}
\int_{V_2}^{\xi X^{1/2-\eps}}
\frac{\dx x}{x^2}
\ll
\frac{X\xi (\log X)^{2} }{V_2\eta^{3}},
\]
and
\[
M_{6} \ll 
\frac{X \xi  (\log X)^{4} }{\eta^{3}}
\int_{\xi X^{1/2-\eps}}^{+\infty}   
\frac{\dx x}{x^{2}}
\ll
\frac{X^{1/2+\eps}(\log X)^{4}}{\eta^{3}}.\]
Hence, recalling $\xi > X^{-1/2+\eps}$, by \eqref{J2-split} and 
the definitions of $V_1$ and $V_2$  we get 
\begin{align}
J_2 
\ll 
\label{J2-estim} 
X\xi  (\log X)^{2} \Bigl(V_1^2  + \frac{ (\log X)^{2} }{V_2\eta^{3}}\Bigr)
+
\frac{(X \xi)^{1-a}}{\eta^{2+a} (\log X)^{b} } .
\end{align}
Choosing $V_1=\eta^{1/2}/\log X$ and $V_2=\log^3 X /\eta^{4}$,
by \eqref{estim-1}-\eqref{J1-behav} and \eqref{J2-estim},  we obtain
\begin{equation}
\label{totale}
\int_{0}^{\infty}  
\vert R(\alpha) \vert^2\, 
\widehat{K}_\eta\Bigl( \frac{\alpha}{\xi}\Bigr) \, \dx \alpha
=
X\xi \log X\xi + \frac{c}{2} X\xi  
+ 
\Odi{\eta X\xi \log X 
+
\frac{(X \xi)^{1-a}}{\eta^{2+a} (\log X)^{b}}}.
\end{equation}
To optimize the error term we choose
$ \eta^{3+a} = (X \xi)^{-a} (\log X)^{-b-1}$, 
so that \eqref{totale} becomes
\begin{equation}
\label{semi-finale}
\int_{0}^{\infty}  
\vert R(\alpha) \vert^2\, 
\widehat{K}_\eta\Bigl( \frac{\alpha}{\xi}\Bigr) \, \dx \alpha
=
X\xi \log X\xi + \frac{c}{2} X\xi  
+ 
\Odi{
\frac{(X \xi)^{3/(3+a)}}{ (\log X)^{(b-a-2)/(3+a)}}}.
\end{equation}

Finally, by \eqref{bounds-1} and \eqref{semi-finale},
we obtain
\[
R(X,\xi) \leq 2X\xi \log X\xi + c X\xi 
+
\Odi{\frac{(X \xi)^{3/(3+a)}}{ (\log X)^{(b-a-2)/(3+a)}}}.
\]
In a similar way we also get that
\[
R(X,\xi) \geq 2X\xi \log X\xi + c X\xi 
+
\Odi{\frac{(X \xi)^{3/(3+a)}}{ (\log X)^{(b-a-2)/(3+a)}}},
\]
and Theorem \ref{Th-rianda} follows.

\renewcommand{\bibliofont}{\normalsize} 

\providecommand{\bysame}{\leavevmode\hbox to3em{\hrulefill}\thinspace}
\providecommand{\MR}{\relax\ifhmode\unskip\space\fi MR }
\providecommand{\MRhref}[2]{%
  \href{http://www.ams.org/mathscinet-getitem?mr=#1}{#2}
}
\providecommand{\href}[2]{#2}

\vskip1cm 
\noindent
\begin{tabular}{l@{\hskip 16mm}l}
Alessandro Languasco               & Alessandro Zaccagnini\\
Universit\`a di Padova     & Universit\`a di Parma\\
Dipartimento di Matematica & Dipartimento di Matematica e Informatica \\
Via Trieste 63                & Parco Area delle Scienze, 53/a \\
35121 Padova, Italy            & 43124 Parma, Italy\\
\emph{e-mail}: languasco@math.unipd.it        & \emph{e-mail}:
alessandro.zaccagnini@unipr.it  
\end{tabular}

\end{document}